\documentclass[a4paper,11pt]{article}
\title{Excluding Induced Subdivisions of the Bull and Related Graphs}
\author{Maria Chudnovsky\thanks{Columbia University, New York, NY 10027, USA. E-mail: mchudnov@columbia.edu. Partially supported by NSF grants DMS-0758364 and DMS-1001091.}, 
Irena Penev\thanks{Columbia University, New York, NY 10027, USA. E-mail: ipenev@math.columbia.edu.}, 
Alex Scott\thanks{Mathematical Institute, University of Oxford, 24-29 St Giles', Oxford OX1 3LB, UK. E-mail: scott@maths.ox.ac.uk.}, 
Nicolas Trotignon\thanks{CNRS, LIP -- ENS Lyon, 15 parvis Ren\'e Descartes, BP 7000, 69342 Lyon cedex 07,
   France.\ E-mail: nicolas.trotignon@ens-lyon.fr. Partially supported by \emph{Agence Nationale de la Recherche} under reference
    \textsc{anr 10 jcjc 0204 01}}} 
\date{June 27, 2011}
\usepackage{amsthm}
\usepackage{amssymb}
\usepackage{amsbsy}
\usepackage{mathrsfs} 
\newtheorem{theorem}{}[section] 
\newtheorem*{decomposition bull}{\ref{decomposition bull}} 
 
\newtheorem*{alloy}{\ref{alloy}} 
\newtheorem*{alloy induction case}{\ref{alloy induction case}} 
\begin{document}
\maketitle 
\noindent 
\begin{abstract} 
\noindent 
For any graph $H$, let ${\rm Forb}^*(H)$ be the class of graphs with
no induced subdivision of $H$. It was conjectured in [A.D. Scott,
Induced trees in graphs of large chromatic number, {\em Journal of 
Graph Theory}, 24:297--311, 1997] that, for every graph $H$, there
is a function $f_H:\mathbb{N} \rightarrow \mathbb{R}$ such that for
every graph $G \in {\rm Forb}^*(H)$, $\chi(G) \leq f_H(\omega(G))$. We
prove this conjecture for several graphs $H$, namely the paw (a
triangle with a pendant edge), the bull (a triangle with two
vertex-disjoint pendant edges), and what we call a ``necklace,'' that
is, a graph obtained from a path by choosing a matching such that no
edge of the matching is incident with an endpoint of the path, and for
each edge of the matching, adding a vertex adjacent to the ends of
this edge.
\end{abstract}

\section{Introduction} 
All graphs in this paper are finite and simple. A {\em clique} (respectively: {\em stable set}) in a graph $G$ is a set of pairwise adjacent (respectively: non-adjacent) vertices in $G$. Given a graph $G$, we denote by $\omega(G)$ the clique number of $G$ (i.e.\ the maximum number of vertices in a clique in $G$), and we denote by $\chi(G)$ the chromatic number of $G$. A class $\mathcal{G}$ of graphs is said to be {\em hereditary} if it is closed under isomorphism and taking induced subgraphs. A hereditary class $\mathcal{G}$ is said to be {\em $\chi$-bounded} if there is a non-decreasing function $f:\mathbb{N} \rightarrow \mathbb{R}$ such that $\chi(G) \leq f(\omega(G))$ for all graphs $G \in \mathcal{G}$; under such circumstances, we say that the class $\mathcal{G}$ is {\em $\chi$-bounded by $f$}, and that $f$ is a {\em $\chi$-bounding function} for $\mathcal{G}$. Given a graph $H$, we say that a graph $G$ is an $H^*$ provided that $G$ is a subdivision of the graph $H$ (in particular, the graph $H$ itself is an $H^*$). Given a graph $H$, we say that a graph $G$ is {\em $H$-free} if it does not contain $H$ as an induced subgraph, and we say that $G$ is {\em $H^*$-free} if it does not contain any subdivision of $H$ as an induced subgraph. We denote by ${\rm Forb}(H)$ the class of all $H$-free graphs, and we denote by ${\rm Forb}^*(H)$ the class of all $H^*$-free graphs. Clearly, ${\rm Forb}(H)$ and ${\rm Forb}^*(H)$ are hereditary classes for every graph $H$. 
\\
\\
Gy\'arf\'as~\cite{gyarfas:conjecture} and Sumner~\cite{sumner:t} independently conjectured that for any tree $T$, the class ${\rm Forb}(T)$ is $\chi$-bounded. The conjecture has been proven for trees of radius $2$ and a few trees of larger radius (see \cite{gyarfas:perfect}, \cite{gst}, \cite{kp}, \cite{kz}, \cite{scott:tree}). Scott~\cite{scott:tree} proved a weakened (``topological'') version of the conjecture: for any tree $T$, the class ${\rm Forb}^*(T)$ is $\chi$-bounded. (Since every forest is an induced subgraph of some tree, this result immediately implies that ${\rm Forb}^*(F)$ is $\chi$-bounded for every forest $F$.) Scott further conjectured that for any graph $H$, the class ${\rm Forb}^*(H)$ is $\chi$-bounded; this generalized a still-open conjecture of Gy\'arf\'as~\cite{gyarfas:perfect}, that the class ${\rm Forb}^*(C_n)$ is $\chi$-bounded for every $n$, where $C_n$ is the chordless cycle of length $n$ (see also \cite{s2}). The aim of this paper is to investigate Scott's conjecture for several particular graphs $H$. 
\\
\\
The {\em paw} is the graph with vertex-set $\{x_1,x_2,x_3,y\}$ and edge-set $\{x_1x_2,x_2x_3,x_3x_1,x_1y\}$. In section~\ref{sec:paw}, we give a structural description of the class ${\rm Forb}^*(paw)$, which we then use to compute the best possible $\chi$-bounding function for the class (see \ref{bound paw}). Together with previously known results, this theorem implies that the class ${\rm Forb}^*(H)$ is $\chi$-bounded for all graphs $H$ on at most four vertices. Indeed, if $H$ is a forest, then the result follows from the result of Scott~\cite{scott:tree} mentioned above. If $H$ is the {\em triangle} (i.e.\ the complete graph on three vertices), then ${\rm Forb}^*(H)$ is the class of all forests. If $H$ is the graph with vertex-set $\{x,y,z,w\}$ and edge-set $\{xy,yz,zx\}$, then any graph $G$ in ${\rm Forb}^*(H)$ can be partitioned into a forest and a graph whose clique number is smaller than $\omega(G)$ (indeed, take any vertex $v$ of $G$, and note that the subgraph of $G$ induced $v$ and its non-neighbors is a forest, while the subgraph of $G$ induced by the neighbors of $v$ has clique number smaller than $\omega(G)$), and consequently, ${\rm Forb}^*(H)$ is $\chi$-bounded by the function $f(n) = 2n$. If $H$ is the {\em diamond} (i.e.\ the graph obtained by deleting an edge from the complete graph on four vertices), then the result follows from a theorem of Trotignon and Vu\v skovi\'c, see \cite{nicolas.kristina:one}. If $H$ is the complete graph on four vertices, Scott's conjecture follows from the work of several authors, see \cite{nicolas:isk4}. Finally, if $H$ is the {\em square} (i.e.\ the chordless cycle on four vertices), then ${\rm Forb}^*(H)$ is the famous class of \emph{chordal} graphs, see~\cite{livre:perfectgraphs}. 
\\
\\
The {\em bull} is the graph with vertex-set $\{x_1,x_2,x_3,y_1,y_2\}$ and edge-set $\{x_1x_2,x_2x_3,x_3x_1,x_1y_1,x_2y_2\}$. In section~\ref{sec:bullDecomp}, we prove a decomposition theorem for bull$^*$-free graphs, see~\ref{decomposition bull}. In section~\ref{sec:bullChiBound}, we use this theorem to prove that the class ${\rm Forb}^*(bull)$ is $\chi$-bounded by the function $f(n) = n^2$, see~\ref{quadratic}. We note that this is the best possible polynomial $\chi$-bounding function for ${\rm Forb}^*(bull)$ in the following sense: there do not exist positive constants $c,r \in \mathbb{R}$, with $r < 2$, such that ${\rm Forb}^*(bull)$ is $\chi$-bounded by the function $f(n) = cn^r$. As ${\rm Forb}^*(bull)$ contains all graphs with no stable set of size three, this follows immediately from a result of Kim \cite{Kim} that the Ramsey number $R(t,3)$ has order of magnitude $\frac{t^2}{\log t}$ (in fact, it is enough that $R(t,3)=t^{2-o(1)}$, which also follows from an earlier result of Erd\H os \cite{erdos:gapii}).
\\
\\
Finally, in section \ref{sec:necklaces}, we consider graphs that we call ``necklaces.'' A necklace is a graph obtained from a path by choosing a matching such that no edge of the matching is incident with an endpoint of the path, and for each edge of the matching, adding a vertex adjacent to the ends of this edge (see section \ref{sec:necklaces} for a more formal definition). We prove that for any given necklace $N$, the class ${\rm Forb}^*(N)$ is $\chi$-bounded by an exponential function (see~\ref{necklace}). We observe that the bull is a special case of a necklace, and so the results of section \ref{sec:necklaces} imply that ${\rm Forb}^*(bull)$ is $\chi$-bounded; however, the $\chi$-bounding function for ${\rm Forb}^*(bull)$ from \ref{quadratic} is polynomial, whereas the one from \ref{necklace} is exponential. Further, we note that for all positive integers $m$, the $m$-edge path, denoted by $P_{m+1}$, is a necklace; furthermore, since any subdivision of an $m$-edge path contains an $m$-edge path as an induced subgraph, we know that ${\rm Forb}(P_{m+1}) = {\rm Forb}^*(P_{m+1})$. Thus, \ref{necklace} implies a result of ~Gy{\'a}rf{\'a}s (see \cite{gyarfas:perfect}) that the class ${\rm Forb}(P_{m+1})$ is $\chi$-bounded by an exponential function (we note, however, that our $\chi$-bounding function is faster growing than that of Gy{\'a}rf{\'a}s). 
\\
\\
We end this section with some terminology and notation that will be used throughout the paper. The vertex-set of a graph $G$ is denoted by $V_G$. Given a vertex $v \in V_G$, $\Gamma_G(v)$ is the set of all neighbors of $v$ in $G$. The complement of $G$ is denoted by $\overline{G}$. Given a set $S \subseteq V_G$, the subgraph of $G$ induced by $S$ is denoted by $G[S]$; if $S = \{v_1,...,v_n\}$, we sometimes write $G[v_1,...,v_n]$ instead of $G[S]$. Given a set $S \subseteq V_G$, we denote by $G \smallsetminus S$ the graph obtained by deleting from $G$ all the vertices in $S$; if $S = \{v\}$, we often write $G \smallsetminus v$ instead of $G \smallsetminus S$. Given a vertex $v \in V_G$ and a set $A \subseteq V_G \smallsetminus \{v\}$, we say that $v$ is {\em complete} (respectively: {\em anti-complete}) to $A$ provided that $v$ is adjacent (respectively: non-adjacent) to every vertex in $A$; we say that $v$ is {\em mixed} on $A$ provided that $v$ is neither complete nor anti-complete to $A$. Given disjoint sets $A,B \subseteq V_G$, we say that $A$ is {\em complete} (respectively: {\em anti-complete}) to $B$ provided that every vertex in $A$ is complete (respectively: anti-complete) to $B$.

\section{Subdivisions of the Paw} 
\label{sec:paw}

In this section, we give a structure theorem for paw$^*$-free graphs (\ref{structure paw}), and then use it to derive the fact that ${\rm Forb}^*(paw)$ is $\chi$-bounded by a linear function (\ref{bound paw}). We first need a definition: a graph is said to be {\em complete multipartite} if its vertex-set can be partitioned into stable sets, pairwise complete to each other. 
\begin{theorem} \label{structure paw} A graph $G$ is paw$^*$-free if and only if each of its components is either a tree, a chordless cycle, or a complete multipartite graph. 
\end{theorem} 
\begin{proof} 
The `if' part is established by routine checking. For the `only if' part, suppose that $G$ is a connected paw$^*$-free graph. Our goal is to show that if $G$ is both triangle-free and square-free, then $G$ is either a tree or a chordless cycle, and otherwise $G$ is a complete multipartite graph. 
\\
\\
Suppose first that $G$ is both triangle-free and square-free. If $G$ contains no cycles, then it is a tree, and we are done. So assume that $G$ does contain a cycle, and let $v_0-v_1-...-v_{k-1}-v_0$ (with the indices in $\mathbb{Z}_k$) be a cycle in $G$ of length as small as possible; note that the minimality of $k$ implies that this cycle is induced, and the fact that $G$ is triangle-free and square-free implies that $k \geq 5$. If $V_G = \{v_0,v_1,...,v_{k-1}\}$, then $G$ is a chordless cycle, and we are done. So assume that $\{v_0,...,v_{k-1}\} \subsetneqq V_G$. Since $G$ is connected, there exists a vertex $v \in V_G \smallsetminus \{v_0,...,v_{k-1}\}$ that has a neighbor in $\{v_0,...,v_{k-1}\}$. Note that $v$ must have at least two neighbors in $\{v_0,v_1,...,v_{k-1}\}$, for otherwise, $G[v,v_0,v_1,...,v_{k-1}]$ would be a paw$^*$. By symmetry, we may assume that for some $i \in \mathbb{Z}_k \smallsetminus \{0\}$, $v$ is complete to $\{v_0,v_i\}$ and anti-complete to $\{v_1,...,v_{i-1}\}$ in $G$. By the minimality of $k$, the cycle $v-v_0-v_1-...-v_i-v$ is of length at least $k$, and so it follows that either $i = k-2$ or $i = k-1$. But then $v-v_i-v_{i+1}-...-v_0-v$ is a (not necessarily induced) cycle of length at most four in $G$, which contradicts the fact that $G$ is triangle-free and square-free. 
\\
\\
It remains to consider the case when $G$ contains a triangle or a square. Let $H$ be an inclusion-wise maximal complete multipartite induced subgraph of $G$ such that $H$ contains a cycle. (The existence of such a graph $H$ follows from the fact that a triangle or a square is itself a complete multipartite graph that contains a cycle.) If $G = H$, then $G$ is complete multipartite, and we are done. So assume that this is not the case. Since $G$ is connected, there exists a vertex $v \in V_G \smallsetminus V_H$ with a neighbor in $V_H$. 
\\
\\
Let $H_1,H_2,...,H_k$ be a partition of $V_H$ into stable sets, pairwise complete to each other. First, we claim that $v$ is not mixed on any set among $H_1,...,H_k$. Suppose otherwise. By symmetry, we may assume that $v$ is adjacent to some $h_1 \in H_1$ and non-adjacent to some $h_1' \in H_1$. Then $v$ is anti-complete to $H_2 \cup ... \cup H_k$, for if $v$ had a neighbor $h \in H_2 \cup ... \cup H_k$, then $G[v,h,h_1,h_1']$ would be a paw. Now, since $H$ contains a cycle, we know that $|H_2 \cup ... \cup H_k| \geq 2$; fix distinct vertices $h,h' \in H_2 \cup ... \cup H_k$. But if $hh'$ is an edge then $G[h,h',h_1,v]$ is a paw, and if $hh'$ is a non-edge then $G[h,h',h_1,h_1',v]$ is a paw$^*$. This proves our claim. Now $v$ is anti-complete to at least two sets among $H_1,...,H_k$ (say $H_1$ and $H_2$), for otherwise, $G[V_H \cup \{v\}]$ would contradict the maximality of $H$. Let $h \in H_3 \cup ... \cup H_k$ be some neighbor of $v$, and fix $h_1 \in H_1$ and $h_2 \in H_2$. Then $G[h_1,h_2,h,v]$ is a paw, which is a contradiction. This completes the argument. 
\end{proof} 
\noindent 
We note that our structure theorem for paw$^*$-free graphs (\ref{structure paw}) is similar to the structure theorem for paw-free graphs (due to Olariu \cite{Olariu}), which states that a graph $G$ is paw-free if and only if every component of $G$ is either triangle-free or complete multipartite. In fact, our proof of \ref{structure paw} could be slightly shortened by using \cite{Olariu}, but in order to keep the section  self-contained, we include an  independent proof. We now turn to  proving that the class ${\rm Forb}^*(paw)$ is $\chi$-bounded by a linear function. 
\begin{theorem} \label{bound paw} ${\rm Forb}^*(paw)$ is $\chi$-bounded by the function $f:\mathbb{N} \rightarrow \mathbb{R}$ defined by $f(2)=3$ and for all $n \neq 2$, $f(n) = n$. 
\end{theorem} 
\begin{proof} 
Let $G \in {\rm Forb}^*(paw)$. We may assume that $G$ is connected (otherwise, we consider the components of $G$ separately). By \ref{structure paw} then, $G$ is either a tree, or a chordless cycle, or a complete multipartite graph, and in each of these cases, we have that $\chi(G) = 3$ or $\chi(G) = \omega(G)$. 
\end{proof} 
\noindent 
It is easy to see that the $\chi$-bounding function given in \ref{bound paw} is the best possible for the class ${\rm Forb}^*(paw)$. Indeed, on the one hand, we have that $\omega(G) \leq \chi(G)$ for every graph $G$, and on the other hand, there exist paw$^*$-free graphs with clique number $2$ and chromatic number $3$ (any chordless cycle of odd length greater than three is such a graph.) 

\section{Decomposing Bull$^*$-Free Graphs} 
\label{sec:bullDecomp} 
In this section, we prove a decomposition theorem for bull$^*$-free graphs. We begin with some definitions. Let $G$ be a graph. A {\em hole} in $G$ is an induced cycle in $G$ of length at least four. An {\em anti-hole} in $G$ is an induced subgraph of $G$ whose complement is a hole in $\overline{G}$. We often denote a hole (respectively: anti-hole) $H$ in $G$ by $h_0-h_1-...-h_k-h_0$, where $V_H = \{h_0,h_1,...,h_k\}$ and $h_0-h_1-...-h_k-h_0$ is an induced cycle in $G$ (respectively: in $\overline{G}$). The {\em length} of a hole or anti-hole is the number of vertices that it contains. An {\em odd hole} (respectively: {\em odd anti-hole}) is a hole (respectively: anti-hole) of odd length. Given a vertex $v \in V_G$ and a set $S \subseteq V_G \smallsetminus \{v\}$, we say that $v$ is a {\em center} (respectively: {\em anti-center}) for $S$ or for $G[S]$ provided that $v$ is complete (respectively: anti-complete) to $S$. We say that $G$ is {\em basic} if it contains neither an odd hole with an anti-center nor an odd anti-hole with an anti-center. A non-empty set $S \subsetneqq V_G$ is said to be a {\em homogeneous set} in $G$ provided that no vertex in $V_G \smallsetminus S$ is mixed on $S$; a homogeneous set $S$ in $G$ is said to be {\em proper} if $|S| \geq 2$. We say that a vertex $v \in V_G$ is a {\em cut-vertex} of $G$ provided that $G \smallsetminus v$ has more components than $G$. Our goal in this section is to prove the following decomposition theorem. 
\begin{theorem} \label{decomposition bull} Let $G \in {\rm Forb}^*(bull)$. Then either $G$ is basic, or it contains a proper homogeneous set or a cut-vertex.
\end{theorem} 
\noindent 
We will need the following result, which is an immediate consequence of $1.4$ from \cite{Maria}. 
\begin{theorem}[Chudnovsky and Safra \cite{Maria}] \label{center} Let $G \in {\rm Forb}^*(bull)$. If $G$ contains an odd hole with a center and an anti-center, or an odd anti-hole with a center and an anti-center, then $G$ has a proper homogeneous set. 
\end{theorem} 
\noindent 
The proof of \ref{decomposition bull} proceeds as follows. We assume that a graph $G \in {\rm Forb}^*(bull)$ is not basic, and then we consider two cases: when $G$ contains an odd anti-hole of length at least seven with an anti-center; and when $G$ contains an odd hole with an anti-center. In the former case, we show that $G$ contains a proper homogeneous set (see~\ref{anti-hole} below). The latter case is more difficult, and our approach is to prove a series of lemmas that describe how vertices that lie outside of our odd hole ``attach'' to this odd hole and to each other, and then to use these results to prove that $G$ contains a proper homogeneous set or a cut-vertex (see~\ref{hole}). Since an anti-hole of length five is also a hole of length five, these two results (\ref{anti-hole} and \ref{hole}) imply \ref{decomposition bull}. 
\begin{theorem} \label{anti-hole} Let $G \in {\rm Forb}^*(bull)$, let $h_0-h_1-...-h_{k-1}-h_0$ (with $k \geq 7$ and the indices in $\mathbb{Z}_k$) be an odd anti-hole in $G$, and set $H = \{h_0,h_1,...,h_{k-1}\}$. Assume that $G$ contains an anti-center for $H$. Then $G$ contains a proper homogeneous set. 
\end{theorem} 
\begin{proof} 
We may assume that $G$ is connected, for otherwise, $G$ contains a proper homogeneous set and we are done. Since $G$ is connected and contains an anti-center for $H$, there exist adjacent $a,a' \in V_G \smallsetminus H$ such that $a$ is anti-center for $H$ and $a'$ has a neighbor in $H$. Our goal is to show that $a'$ is a center for $H$, for then we are done by \ref{center}. 
\\
\\
First, we claim that there is no index $i \in \mathbb{Z}_k$ such that $a'$ is anti-complete to $\{h_i,h_{i+1}\}$. Suppose otherwise. Since $a'$ has a neighbor in $H$, we may assume by symmetry that $a'$ is adjacent to $h_0$ and anti-complete to $\{h_1,h_2\}$. But then if $a'h_4$ is an edge, then $G[h_0,h_1,h_4,a,a']$ is a bull; and if $a'h_4$ is a non-edge, then $G[h_0,h_1,h_2,h_4,a']$ is a bull. This proves our claim. 
\\
\\
Next, since $H$ has an odd number of vertices, there exists some $i \in \mathbb{Z}_k$ such that $a'$ is either complete or anti-complete to $\{h_i,h_{i+1}\}$; by what we just showed, the latter is impossible, and so the former must hold. Now, if $a'$ is not a center for $H$, then we may assume by symmetry that $a'$ is non-adjacent to $h_0$ and complete to $\{h_1,h_2\}$; but then $a'h_{k-1}$ is an edge (because $a'$ is not anti-complete to $\{h_{k-1},h_0\}$), and so $G[h_0,h_2,h_{k-1},a,a']$ is a bull. Thus, $a'$ is a center for $H$, which completes the argument. 
\end{proof} 
\noindent 
For the remainder of this section, we focus on graphs in ${\rm Forb}^*(bull)$ that contain an odd-hole with an anti-center. We begin with some definitions. Let $G$ be a graph, let $h_0-h_1-...-h_{k-1}-h_0$ (with $k \geq 5$ and the indices in $\mathbb{Z}_k$) be a hole in $G$, let $H = \{h_0,h_1,...,h_{k-1}\}$, and let $v \in V_G \smallsetminus H$. Then for all $i \in \mathbb{Z}_k$: 
\begin{itemize} 
\item $v$ is a {\em leaf for $H$ at $h_i$} if $v$ is adjacent to $h_i$ and anti-complete to $H \smallsetminus \{h_i\}$; 
\item $v$ is a {\em star for $H$ at $h_i$} if $v$ is complete to $H \smallsetminus \{h_i\}$ and non-adjacent to $h_i$; 
\item $v$ is an {\em adjacent clone for $H$ at $h_i$} if $v$ is complete to $\{h_{i-1},h_i,h_{i+1}\}$ and anti-complete to $H \smallsetminus \{h_{i-1},h_i,h_{i+1}\}$; 
\item $v$ is a {\em non-adjacent clone for $H$ at $h_i$} if $v$ is complete to $\{h_{i-1},h_{i+1}\}$ and anti-complete to $H \smallsetminus \{h_{i-1},h_{i+1}\}$; 
\item $v$ is a {\em clone for $H$ at $h_i$} if $v$ is an adjacent clone or a non-adjacent clone for $H$ at $h_i$. 
\end{itemize} 
We say that $v$ is a {\em leaf} (respectively: {\em star}, {\em adjacent clone}, {\em non-adjacent clone}, {\em clone}) for $H$ if there exists some $i \in \mathbb{Z}_k$ such that $v$ is a leaf (respectively: star, adjacent clone, non-adjacent clone, clone) for $H$ at $h_i$. If $|H| = k$ is odd, then we say that a vertex $v \in V_G \smallsetminus H$ is {\em appropriate} for $H$ or for $G[H]$ provided that one of the following holds: 
\begin{itemize} 
\item $v$ is a center for $H$; 
\item $v$ is an anti-center for $H$; 
\item $v$ is a leaf for $H$; 
\item $v$ is an adjacent clone for $H$; 
\item $v$ is a non-adjacent clone for $H$ and $|H| = 5$; 
\item $v$ is a star for $H$ and $|H| = 5$. 
\end{itemize} 
\begin{theorem} \label{appropriate} Let $G \in {\rm Forb}^*(bull)$, let $h_0-h_1-...-h_{k-1}-h_0$ (with $k \geq 5$ and the indices in $\mathbb{Z}_k$) be an odd hole in $G$, and set $H = \{h_0,h_1,...,h_{k-1}\}$. Then every vertex in $V_G \smallsetminus H$ is appropriate for $H$. 
\end{theorem} 
\begin{proof} 
Fix $v \in V_G \smallsetminus H$. We may assume that $v$ has at least two neighbors and at least one non-neighbor in $H$, for otherwise, $v$ is a center, an anti-center, or a leaf for $H$, and we are done. 
\\
\\
Suppose first that $v$ has two adjacent neighbors in $H$. Fix a path $h_i-h_{i+1}-...-h_j$ of maximum length in $G[H \cap \Gamma_G(v)]$; set $P = \{h_i,h_{i+1},...,h_j\}$. Note first that $|P| \geq 3$, for otherwise, we would have that $j = i+1$, and then $G[v,h_{i-1},h_i,h_{i+1},h_{i+2}]$ would be a bull. Now, we claim that $v$ is anti-complete to $H \smallsetminus P$. Suppose otherwise. Fix $h_l \in H \smallsetminus P$ such that $vh_l$ is an edge; by the maximality of $P$, we know that $l \notin \{i-1,j+1\}$. Since neither $G[v,h_{i-1},h_i,h_{i+1},h_l]$ nor $G[v,h_{j-1},h_j,h_{j+1},h_l]$ is bull, we get that $l = i-2 = j+2$, and consequently, that $|H| = |P|+3$. Since $|H|$ is odd and $|P| \geq 3$, this means that $|P| \geq 4$, and so $G[v,h_{i-1},h_i,h_{i+1},h_{i+3}]$ is a bull, which is a contradiction. It follows that $v$ is anti-complete to $H \smallsetminus P$. Now, if $|P| = 3$, then $v$ is an adjacent clone for $H$ at $h_{i+1}$, and we are done. So assume that $|P| \geq 4$. Since $G[v,h_{i-1},h_i,h_{i+1},h_{i+3}]$ is not a bull, $h_{i+3}$ is adjacent to $h_{i-1}$, and so $|H| = 5$ and $v$ is a star for $H$ at $h_{i-1}$. 
\\
\\
Suppose now that $H \cap \Gamma_G(v)$ is a stable set. Fix distinct $i,j \in \mathbb{Z}_k$ such that $v$ is complete to $\{h_i,h_j\}$ and the path $h_i-h_{i+1}-...-h_j$ is as short as possible (in particular, $v$ is non-adjacent to the interior vertices of the path). Since the neighbors of $v$ in $H$ are pairwise non-adjacent, and $v$ is complete to $\{h_i,h_j\}$, we know that $v$ is anti-complete to $\{h_{i-1},h_{j+1}\}$. Since $G[v,h_{i-1},h_i,h_{i+1},...,h_j,h_{j+1}]$ is not a bull$^*$, this implies that either $h_{i-1} = h_{j+1}$, or $h_{i-1}h_{j+1}$ is an edge, and in either case, $v$ is anti-complete to $H \smallsetminus \{h_i,h_j\}$. We now know that the path $h_j-h_{j+1}-...-h_i$ has at most three edges and that $v$ is adjacent to the ends of this path and non-adjacent to its interior vertices. The minimality of the path $h_i-h_{i+1}-...-h_j$ then implies that $|H| \leq 6$. Since $|H|$ is odd and $|H| \geq 5$, it follows that $|H| = 5$. The minimality of the path $h_i-h_{i+1}-...-h_j$ now implies that $v$ is a non-adjacent clone for $H$ at $h_{i+1}$. This completes the argument. 
\end{proof} 
\noindent 
Given a graph $G$ with a hole $h_0-h_1-...-h_{k-1}-h_0$ (with $k \geq 5$ and the indices in $\mathbb{Z}_k$), and setting $H = \{h_0,h_1,...,h_{k-1}\}$, we let $A_H$ denote the set of all anti-centers for $H$ in $G$, and for all $i \in \mathbb{Z}_k$: 
\begin{itemize} 
\item we let $L_H^i$ denote the set of all leaves for $H$ at $h_i$; 
\item we let $N_H^i$ denote the set of all non-adjacent clones for $H$ at $h_i$; 
\item we let $C_H^i$ denote the set of all adjacent clones for $H$ at $h_i$; 
\item we let $S_H^i$ denote the set of all stars for $H$ at $h_i$. 
\end{itemize} 
\begin{theorem} \label{anti-center leaf} Let $G \in {\rm Forb}^*(bull)$, let $h_0-h_1-...-h_{k-1}-h_0$ (with $k \geq 5$ and the indices in $\mathbb{Z}_k$) be an odd hole in $G$, and set $H = \{h_0,h_1,...,h_{k-1}\}$. Assume that $G$ contains an anti-center for $H$, and that $G$ does not contain a proper homogeneous set. Then there exists an index $i \in \mathbb{Z}_k$ such that all of the following hold: 
\begin{itemize} 
\item[(i)] $L_H^i \neq \emptyset$, and for all $j \in \mathbb{Z}_k \smallsetminus \{i\}$, $L_H^j = \emptyset$; 
\item[(ii)] $A_H$ is not anti-complete to $L_H^i$; 
\item[(iii)] $A_H$ is anti-complete to $V_G \smallsetminus (A_H \cup L_H^i)$. 
\end{itemize} 
\end{theorem} 
\begin{proof} 
First, since $G$ does not contain a proper homogeneous set and $|V_G| \geq 3$, we know that $G$ is connected. Further, since $G$ does not contain a proper homogeneous set and contains an anti-center for $H$, \ref{center} implies that $G$ does not contain a center for $H$.
\\
\\
Now, we claim that every vertex in $V_G \smallsetminus (H \cup A_H)$ that has a neighbor in $A_H$ is a leaf for $H$. Suppose otherwise; fix adjacent $v \in V_G \smallsetminus (H \cup A_H)$ and $a \in A_H$ such that $v$ is not a leaf for $H$. Since $v$ is appropriate for $H$ (by \ref{appropriate}), and since $v$ is not a leaf, or a center, or an anti-center for $H$, we know that $v$ is either a star, or an adjacent clone, or a non-adjacent clone for $H$. Suppose first that $v$ is a star or an adjacent clone for $H$. Then there exists an index $i \in \mathbb{Z}_k$ such that $v$ is complete to $\{h_i,h_{i+1}\}$ and non-adjacent to $h_{i+2}$; but now $G[a,v,h_i,h_{i+1},h_{i+2}]$ is a bull. Suppose now that $v$ is a non-adjacent clone for $H$. Then there exists an index $i \in \mathbb{Z}_k$ such that $v$ is complete to $\{h_{i-1},h_{i+1}\}$ and anti-complete to $\{h_i,h_{i+2}\}$; but now $G[a,v,h_{i-1},h_i,h_{i+1},h_{i+2}]$ is a bull$^*$. This proves our claim. 
\\
\\
Since $G$ is connected and $A_H$ is non-empty, what we just showed implies that there exists an index $i \in \mathbb{Z}_k$ such that $L_H^i$ is non-empty and is not anti-complete to $A_H$. The only thing left to show is that $L_H^j = \emptyset$ for all $j \in \mathbb{Z}_k \smallsetminus \{i\}$. Suppose otherwise. Fix some $j \in \mathbb{Z}_k \smallsetminus \{i\}$ such that $L_H^j \neq \emptyset$. First, note that $L_H^j$ is complete to $L_H^i$, for if some $l_i \in L_H^i$ and $l_j \in L_H^j$ were non-adjacent, $G[H \cup \{l_i,l_j\}]$ would be a bull$^*$. By symmetry and the fact that $|H|$ is odd, we may assume that the path $h_i-h_{i+1}-...-h_j$ is shorter than the path $h_j-h_{j+1}-...-h_i$; since $|H| \geq 5$, this means that $i-1 \notin \{j,j+1\}$. Note furthermore that $j \neq i+1$, for otherwise, we fix some $l_i \in L_H^i$ and $l_{i+1} \in L_H^{i+1}$ and note that $G[l_i,l_{i+1},h_{i-1},h_i,h_{i+1},h_{i+2}]$ is a bull$^*$. Next, fix an anti-center $a$ for $H$ such that $a$ is adjacent to some $l_i \in L_H^i$. Fix $l_j \in L_H^j$. But then if $al_j$ is an edge, $G[a,l_i,l_j,h_i,h_j]$ is a bull; and if $al_j$ is a non-edge, then $G[a,l_i,l_j,h_{i-1},h_i,h_{i+1},...,h_{j-1},h_j]$ is a bull$^*$. This completes the argument. 
\end{proof} 
\begin{theorem} \label{bull cleaning} Let $G \in {\rm Forb}^*(bull)$, let $h_0-h_1-...-h_{k-1}-h_0$ (with $k \geq 5$ and the indices in $\mathbb{Z}_k$) be an odd hole in $G$, and set $H = \{h_0,h_1,...,h_{k-1}\}$. Assume that $G$ contains an anti-center for $H$, and that $G$ does not contain a proper homogeneous set. Then there exists an index $i \in \mathbb{Z}_k$ such that $V_G = H \cup A_H \cup L_H^i \cup S_H^i \cup \bigcup_{j \in \mathbb{Z}_k} (N_H^j \cup C_H^j)$, where $L_H^i$ is non-empty, $L_H^i$ is anti-complete to $S_H^i$, and if $k \geq 7$, then $S_H^i$ and $\bigcup_{j \in \mathbb{Z}_k} N_H^j$ are empty. 
\end{theorem} 
\begin{proof} 
If $k \geq 7$, then the result is immediate from \ref{center}, \ref{appropriate}, and \ref{anti-center leaf}. So assume that $k = 5$. By \ref{center}, \ref{appropriate}, and \ref{anti-center leaf}, we know that $V_G = H \cup A_H \cup L_H^i \cup \bigcup_{j \in \mathbb{Z}_5} (S_H^j \cup N_H^j \cup C_H^j)$, with $L_H^i \neq \emptyset$, for some $i \in \mathbb{Z}_5$. We need to show that $S_H^j = \emptyset$ for all $j \in \mathbb{Z}_5 \smallsetminus \{i\}$, and that $L_H^i$ is anti-complete to $S_H^i$. 
\\
\\
We first show that $S_H^j = \emptyset$ for all $j \in \mathbb{Z}_5 \smallsetminus \{i\}$. By symmetry, it suffices to show that $S_H^{i+1}$ and $S_H^{i+2}$ are empty. Fix some $l_i \in L_H^i$. Suppose first that $S_H^{i+1} \neq \emptyset$, and fix $s_{i+1} \in S_H^{i+1}$. But then if $s_{i+1}l_i$ is an edge, then $G[l_i,s_{i+1},h_{i-2},h_i,h_{i+1}]$ is a bull; and if $s_{i+1}l_i$ is a non-edge, then $G[l_i,s_{i+1},h_{i-1},h_i,h_{i+2}]$ is a bull. Thus, $S_H^{i+1} = \emptyset$. Suppose now that $S_H^{i+2} \neq \emptyset$, and fix $s_{i+2} \in S_H^{i+2}$. But then if $s_{i+2}l_i$ is an edge, then $G[s_{i+2},l_i,h_{i-2},h_{i-1},h_{i+2}]$ is a bull; and if $s_{i+2}l_i$ is a non-edge, then $G[s_{i+2},l_i,h_i,h_{i+1},h_{i+2}]$ is a bull. Thus, $S_H^{i+2} = \emptyset$. 
\\
\\
It remains to show that $L_H^i$ is anti-complete to $S_H^i$. Suppose otherwise. By \ref{anti-center leaf}, $A_H$ is not anti-complete to $L_H^i$, and $A_H$ is anti-complete to $H \cup S_H^i$. We first note that every vertex in $L_H^i$ is anti-complete to at least one of $A_H$ and $S_H^i$, for otherwise, we fix some $l_i \in L_H^i$, $s_i \in S_H^i$, and $a \in A_H$ such that $l_i$ is adjacent to both $s_i$ and $a$, and we observe that $G[l_i,s_i,a,h_{i-1},h_i,h_{i+2}]$ is a bull$^*$. Now, fix some adjacent $l_i \in L_H^i$ and $s_i \in S_H^i$. By what we just showed, $l_i$ is anti-complete to $A_H$. Since $A_H$ is not anti-complete to $L_H^i$, there exist adjacent $a \in A_H$ and $l_i' \in L_H^i \smallsetminus \{l_i\}$. Since $l_i' \in L_H^i$ has a neighbor in $A_H$, we know that $l_i'$ is anti-complete to $S_H^i$, and in particular, that $l_i's_i$ is a non-edge. But now if $l_il_i'$ is an edge, then $G[l_i,l_i',a,s_i,h_i]$ is a bull; and if $l_il_i'$ is a non-edge, then $G[l_i,l_i',s_i,h_{i-1},h_i,h_{i+2}]$ is a bull$^*$. This completes the argument. 
\end{proof} 
\begin{theorem} \label{no clone} Let $G \in {\rm Forb}^*(bull)$, let $h_0-h_1-...-h_{k-1}-h_0$ (with $k \geq 5$ and the indices in $\mathbb{Z}_k$) be an odd hole in $G$, and set $H = \{h_0,h_1,...,h_{k-1}\}$. Assume that $G$ contains an anti-center for $H$, and that $G$ does not contain a proper homogeneous set. Then there exists an index $i \in \mathbb{Z}_k$ such that $V_G = H \cup A_H \cup L_H^i \cup S_H^i$, where $L_H^i$ is non-empty, $L_H^i$ is anti-complete to $S_H^i$, and if $k \geq 7$, then $S_H^i$ is empty. 
\end{theorem} 
\begin{proof} 
By \ref{bull cleaning}, we just need to show that $N_H^j \cup C_H^j = \emptyset$ for all $j \in \mathbb{Z}_k$. It suffices to show that for all $j \in \mathbb{Z}_k$, $\{h_j\} \cup N_H^j \cup C_H^j$ is a homogeneous set in $G$, for then the fact that $G$ contains no proper homogeneous set will imply that $\{h_j\} \cup N_H^j \cup C_H^j$ is a singleton, and therefore, that $N_H^j \cup C_H^j = \emptyset$. 
\\
\\
Fix $j \in \mathbb{Z}_k$, and suppose that $\{h_j\} \cup N_H^j \cup C_H^j$ is not a homogeneous set in $G$. Fix some $v \in V_G \smallsetminus (\{h_j\} \cup N_H^j \cup C_H^j)$ such that $v$ is mixed on $\{h_j\} \cup N_H^j \cup C_H^j$. Clearly, $v \notin H$. Fix some $c_j,c_j' \in \{h_j\} \cup N_H^j \cup C_H^j$ such that $v$ is adjacent to $c_j$ and non-adjacent to $c_j'$. Set $\hat{H} = (H \smallsetminus \{h_j\}) \cup \{c_j\}$ and $\hat{H}' = (H \smallsetminus \{h_j\}) \cup \{c_j'\}$. Then $G[\hat{H}]$ and $G[\hat{H}']$ are both odd holes of length $k$. Next, by \ref{anti-center leaf}, $A_H$ is anti-complete to $\{c_j,c_j'\}$, and so since $A_H$ is non-empty, $G$ contains an anti-center for both $\hat{H}$ and $\hat{H}'$; thus, \ref{bull cleaning} applies to both $\hat{H}$ and $\hat{H}'$. This, together with the fact that $v$ has exactly one more neighbor in $\hat{H}$ than in $\hat{H}'$, implies that either: 
\begin{itemize} 
\item[(a)] $v$ is a leaf for $\hat{H}$ and an anti-center for $\hat{H}'$; or 
\item[(b)] $k = 5$ and one of the following holds: 
\begin{itemize} 
\item[(b1)] $v$ is a non-adjacent clone for $\hat{H}$ and a leaf for $\hat{H}'$; 
\item[(b2)] $v$ is an adjacent clone for $\hat{H}$ and a non-adjacent clone for $\hat{H}'$; 
\item[(b3)] $v$ is a star for $\hat{H}$ and an adjacent clone for $\hat{H}'$. 
\end{itemize} 
\end{itemize} 
Suppose that (a) holds. Since $v$ is adjacent to $c_j$, $v$ is a leaf for $\hat{H}$ at $c_j$. But now if $c_jc_j'$ is an edge, then $G[v,c_j,c_j',h_{j+1},h_{j+2}]$ is a bull; and if $c_jc_j'$ is a non-edge, then $G[v,c_j,c_j',h_{j-1},h_{j+1},h_{j+2}]$ is a bull$^*$. From now on, we assume that (b) holds, and so $k = 5$. 
\\
\\
Suppose first that (b1) holds. Since $v$ is a non-adjacent clone for $\hat{H}$ and is adjacent to $c_j$, we know that $v$ is a non-adjacent clone for $\hat{H}$ at either $h_{j-1}$ or at $h_{j+1}$; by symmetry, we may assume that $v$ is a non-adjacent clone for $\hat{H}$ at $h_{j+1}$. But now if $c_jc_j'$ is an edge, then $G[v,c_j,c_j',h_{j-2},h_{j-1}]$ is a bull; and if $c_jc_j'$ is a non-edge, then $G[v,c_j,c_j',h_{j-2},h_{j-1},h_{j+1}]$ is a bull$^*$. 
\\
\\
Suppose next that (b2) holds. Since $v$ is a clone for both $\hat{H}$ and $\hat{H}'$, and since $v$ is adjacent to $c_j$ and non-adjacent to $c_j'$, it is easy to see that $v$ is an adjacent clone for $\hat{H}$ at $c_j$ and a non-adjacent clone for $\hat{H}'$ at $c_j'$. But now $v$ is a clone for $H$ at $h_j$, contrary to the fact that $v \in V_G \smallsetminus (\{h_j\} \cup N_H^j \cup C_H^j)$. 
\\
\\
Suppose finally that (b3) holds. Since $v$ is adjacent to $c_j$ and non-adjacent to $c_j'$, it is easy to see that $v$ is a star for $\hat{H}$ at either $h_{j-1}$ or $h_{j+1}$; by symmetry, we may assume that $v$ is a star for $\hat{H}$ at $h_{j+1}$. Since \ref{bull cleaning} applies to $\hat{H}$, it follows that $G$ contains a leaf $l_{j+1}$ for $\hat{H}$ at $h_{j+1}$, and that $l_{j+1}$ is non-adjacent to $v$. Since $l_{j+1}$ is appropriate for $\hat{H}'$, it is non-adjacent to $c_j'$. But now if $c_jc_j'$ is an edge, then $G[v,c_j,c_j',l_{j+1},h_{j+1}]$ is a bull; and if $c_jc_j'$ is a non-edge, then $G[v,c_j,c_j',h_{j-1},h_{j+2}]$ is a bull. This completes the argument. 
\end{proof} 
\begin{theorem} \label{hole} Let $G \in {\rm Forb}^*(bull)$, let $h_0-h_1-...-h_{k-1}-h_0$ (with $k \geq 5$ and the indices in $\mathbb{Z}_k$) be an odd hole in $G$, and set $H = \{h_0,h_1,...,h_{k-1}\}$. Assume that $G$ contains an anti-center for $H$. Then $G$ contains a proper homogeneous set or a cut-vertex. 
\end{theorem} 
\begin{proof} 
We assume that $G$ does not contain a proper homogeneous set and show that it contains a cut-vertex. By \ref{no clone}, there exists an index $i \in \mathbb{Z}_k$ such that $V_G = H \cup A_H \cup L_H^i \cup S_H^i$ and $L_H^i$ is non-empty and anti-complete to $S_H^i$. Now, by \ref{anti-center leaf}, $A_H$ is anti-complete to $S_H^i$. Thus, $A_H \cup L_H^i$ is anti-complete to $(H \smallsetminus \{h_i\}) \cup S_H^i$. Since $V_G = H \cup A_H \cup L_H^i \cup S_H^i$, and since $h_i$ has neighbors both in $L_H^i$ and in $H \smallsetminus \{h_i\}$, it follows that $h_i$ is a cut-vertex of $G$. 
\end{proof} 
\noindent 
We now restate and prove \ref{decomposition bull}, the main result of this section. 
\begin{decomposition bull} Let $G \in {\rm Forb}^*(bull)$. Then either $G$ is basic, or it contains a proper homogeneous set or a cut-vertex. 
\end{decomposition bull} 
\begin{proof} 
Since an anti-hole of length five is also a hole of length five, the result is immediate from \ref{anti-hole} and \ref{hole}. 
\end{proof}

\section{A $\chi$-Bounding Function for ${\rm Forb}^*(bull)$} 
\label{sec:bullChiBound}

In this section, we use \ref{decomposition bull} to prove that the class ${\rm Forb}^*(bull)$ is $\chi$-bounded by the function $f(n) = n^2$. We begin with some definitions. Given graphs $G_1$ and $G_2$ with $V_{G_1} \cap V_{G_2} = \{u\}$, we say that a graph $G$ is obtained by {\em gluing $G_1$ and $G_2$ along $u$} provided that the following hold: 
\begin{itemize} 
\item $V_G = V_{G_1} \cup V_{G_2}$; 
\item for all $i \in \{1,2\}$, $G[V_{G_i}] = G_i$; 
\item $V_{G_1} \smallsetminus \{u\}$ is anti-complete to $V_{G_2} \smallsetminus \{u\}$ in $G$. 
\end{itemize} 
We observe that if a graph $G$ has a cut-vertex, then $G$ is obtained by gluing smaller graphs (i.e.\ graphs that have strictly fewer vertices than $G$) along a vertex. 
\\
\\
Given graphs $G_1$ and $G_2$ with disjoint vertex-sets, a vertex $u \in V_{G_1}$, and a graph $G$, we say that $G$ is obtained by {\em substituting} $G_2$ for $u$ in $G_1$ provided that the following hold:
\begin{itemize} 
\item $V_G = (V_{G_1} \smallsetminus \{u\}) \cup V_{G_2}$; 
\item $G[V_{G_1} \smallsetminus \{u\}] = G_1 \smallsetminus u$; 
\item $G[V_{G_2}] = G_2$; 
\item for all $v \in V_{G_1} \smallsetminus \{u\}$, if $v$ is adjacent (respectively: non-adjacent) to $u$ in $G_1$, then $v$ is complete (respectively: anti-complete) to $V_{G_2}$ in $G$. 
\end{itemize} 
Under these circumstances, we also say that $G$ is obtained by {\em substitution} from $G_1$ and $G_2$. We note that if a graph $G$ has a proper homogeneous set, then it is obtained by substitution from smaller graphs. 
\\
\\
We say that a graph $G$ is {\em perfect} if for every induced subgraph $H$ of $G$, $\chi(H) = \omega(H)$. We now state two results about perfect graphs that we will need in this section. 
\begin{theorem}[Chudnovsky, Robertson, Seymour, and Thomas \cite{SPGT}] \label{perfect} A graph $G$ is perfect if and only if it contains no odd holes and no odd anti-holes. 
\end{theorem} 
\begin{theorem}[Lov\'asz \cite{Lovasz}] \label{replication} Let $G_1$ and $G_2$ be perfect graphs with disjoint vertex-sets, and let $u \in V_{G_1}$. Let $G$ be the graph obtained by substituting $G_2$ for $u$ in $G_1$. Then $G$ is perfect. 
\end{theorem} 
\noindent 
We note that \ref{perfect} is called the strong perfect graph theorem, and \ref{replication} is called the replication lemma. 
\\
\\
In this paper, a {\em weighted graph} is a graph $G$ such that each vertex $v \in V_G$ is assigned a positive integer called its {\em weight} and denoted by $w_v$. The {\em weight} of a non-empty set $S \subseteq V_G$ is the sum of weights of the vertices in $S$. We denote by $W_G$ the weight of a clique of maximum weight in $G$. Given an induced subgraph $H$ of $G$, and a vertex $v \in V_G$, we say that $H$ {\em covers} $v$ provided that $v \in V_H$. We now prove a technical lemma, which we then use to prove the main result of this section. 
\begin{theorem} Let $G \in {\rm Forb}^*(bull)$ be a weighted graph. Then there exists a family $\mathcal{P}_G$ of at most $W_G$ perfect induced subgraphs of $G$ such that for every vertex $v \in V_G$, at least $w_v$ members of $\mathcal{P}_G$ cover $v$. 
\end{theorem} 
\begin{proof} 
We assume inductively that the claim holds for graphs with fewer than $|V_G|$ vertices. By \ref{decomposition bull}, we know that either $G$ is basic, or $G$ contains a proper homogeneous set, or $G$ contains a cut-vertex. 
\\
\\
Suppose first that $V_G$ is basic. Fix $u \in V_G$ such that $w_u$ is maximal. Let $A$ be the set of all neighbors of $u$ in $G$, and let $B$ be the set of all non-neighbors of $u$ in $G$. Since $G$ is basic, and $u$ is an anti-center for $B$, we know that $G[B]$ contains no odd holes and no odd anti-holes. Since $u$ is anti-complete to $B$, it follows that $G[B \cup \{u\}]$ contains no odd holes and no odd anti-holes, and so by the strong perfect graph theorem (\ref{perfect}), $G[B \cup \{u\}]$ is perfect. Let $\mathcal{P}_B$ be the family consisting of $w_u$ copies of the perfect graph $G[B \cup \{u\}]$. Note that by the maximality of $w_u$, every vertex $v \in B \cup \{u\}$ is covered by at least $w_v$ graphs in $\mathcal{P}_B$. If $A = \emptyset$ (so that $V_G = B \cup \{u\}$), then we set $\mathcal{P}_G = \mathcal{P}_B$, and we are done. So assume that $A \neq \emptyset$. Now by the induction hypothesis, there exists a family $\mathcal{P}_A$ of at most $W_{G[A]}$ perfect induced subgraphs of $G[A]$ such that each vertex $v \in A$ is covered by at least $w_v$ graphs in $\mathcal{P}_A$. Since $u$ is complete to $A$, we have that $w_u+W_{G[A]} \leq W_G$. Since the family $\mathcal{P}_B$ contains exactly $w_u$ graphs, it follows that the family $\mathcal{P}_G = \mathcal{P}_A \cup \mathcal{P}_B$ contains at most $W_G$ graphs, and by construction, every vertex $v \in V_G$ is covered by at least $w_v$ graphs in $\mathcal{P}_G$. 
\\
\\
Suppose now that $G$ contains a proper homogeneous set; let $S$ be a proper homogeneous set in $G$, let $A$ be the set of all vertices in $V_G$ that are complete to $S$, and let $B$ be the set of all vertices in $V_G$ that are anti-complete to $S$. Let $H$ be the graph whose vertex-set is $\{s\} \cup A \cup B$, with $H[A \cup B] = G[A \cup B]$, and $s$ complete to $A$ and anti-complete to $B$ in $H$. We turn $H$ into a weighted graph by letting the vertices in $A \cup B$ have the same weights in $H$ as they do in $G$, and setting $w_s = W_{G[S]}$. Clearly, $W_H = W_G$. Using the induction hypothesis, we let $\mathcal{P}_H$ be a family of at most $W_H = W_G$ perfect induced subgraphs of $H$ such that every vertex $v \in V_H$ is covered by at least $w_v$ graphs in $\mathcal{P}_H$, and we let $\mathcal{P}_{G[S]}$ be the family of at most $W_{G[S]} = w_s$ perfect inducted subgraphs of $G[S]$ such that every vertex $v \in S$ is covered by at least $w_v$ graphs in $\mathcal{P}_{G[S]}$. We may assume that the vertex $s$ is covered by exactly $w_s$ graphs in $\mathcal{P}_H$; let $P_1,...,P_{w_s}$ be the graphs in $\mathcal{P}_H$ covering $s$, and let $\mathcal{P}_H' = \mathcal{P}_H \smallsetminus \{P_1,...,P_{w_s}\}$. We may assume that $\mathcal{P}_{G[S]}$ contains exactly $W_{G[S]} = w_s$ graphs; say $\mathcal{P}_{G[S]} = \{Q_1,...,Q_{w_s}\}$. Now, for each $i \in \{1,...,w_s\}$, let $P_i'$ be the graph obtained by substituting the graph $Q_i$ for $s$ in $P_i$; by the replication lemma (\ref{replication}), the graph $P_i'$ is perfect for all $i \in \{1,...,w_s\}$. We then set $\mathcal{P}_G = \{P_1',...,P_{w_s}'\} \cup \mathcal{P}_H'$. By construction, $\mathcal{P}_G$ is a family of at most $W_G$ perfect induced subgraphs of $G$ such that for every vertex $v \in V_G$, at least $w_v$ members of $\mathcal{P}_G$ cover $v$. 
\\
\\
Suppose finally that $G$ contains a cut-vertex. Then there exist $u \in V_G$ and $C_1,C_2 \subseteq V_G \smallsetminus \{u\}$ such that $V_G = \{u\} \cup C_1 \cup C_2$, where $C_1$ and $C_2$ are non-empty, disjoint, and anti-complete to each other. For $i \in \{1,2\}$, let $G_i = G[C_i \cup \{u\}]$. (Note that $G$ is obtained by gluing $G_1$ and $G_2$ along $u$.) Using the induction hypothesis, for each $i \in \{1,2\}$, we get a family $\mathcal{P}_{G_i}$ of at most $W_{G_i}$ perfect induced subgraphs of $G_i$ such that each vertex $v \in V_{G_i}$ is covered by at least $w_v$ graphs in $\mathcal{P}_{G_i}$. We may assume that for all $i \in \{1,2\}$, $\mathcal{P}_{G_i}$ contains exactly $W_{G_i}$ graphs, and that $u_i$ is covered by exactly $w_{u_i}$ graphs in $\mathcal{P}_{G_i}$. By symmetry, we may assume that $W_{G_1} \leq W_{G_2}$. For each $i \in \{1,2\}$, let $P_1^i,...,P_{w_u}^i$ be the graphs in $\mathcal{P}_{G_i}$ covering $u$, let $P_{w_u+1}^i,...,P_{W_{G_1}}^i$ be $W_{G_1}-w_u$ graphs in $\mathcal{P}_{G_i}$ that do not cover $u$, and let $P_{W_{G_1}+1}^2,...,P_{W_{G_2}}^2$ be the remaining $W_{G_2}-W_{G_1}$ graphs in $\mathcal{P}_{G_2}$. Now, for all $j \in \{1,...,w_u\}$, let $P_j$ be the graph obtained by gluing $P_j^1$ and $P_j^2$ along $u$; for all $j \in \{w_u+1,....,W_{G_1}\}$, let $P_j$ be the disjoint union of $P_j^1$ and $P_j^2$; and for all $j \in \{W_{G_1}+1,...,W_{G_2}\}$, let $P_j = P_j^2$. It is easy to see that $P_j$ is perfect for all $j \in \{1,...,W_{G_2}\}$. Now set $\mathcal{P}_G = \{P_1,...,P_{W_{G_2}}\}$. Since $W_G = \max\{W_{G_1},W_{G_2}\} = W_{G_2}$, $\mathcal{P}_G$ is a family of at most $W_G$ perfect induced subgraphs of $G$ such that for every vertex $v \in V_G$, at least $w_v$ members of $\mathcal{P}_G$ cover $v$. 
\end{proof} 
\begin{theorem} \label{quadratic} The class ${\rm Forb}^*(bull)$ is $\chi$-bounded by the function $f(n) = n^2$. 
\end{theorem} 
\begin{proof} 
Let $G \in {\rm Forb}^*(bull)$. Using $4.3$, we obtain a family $\mathcal{P}$ of at most $\omega(G)$ perfect induced subgraphs of $G$ such that each vertex in $V_G$ is covered by at least one graph in $\mathcal{P}$. Clearly, we may assume that each vertex in $V_G$ is covered by exactly one graph in $\mathcal{P}$. Since the graphs in $\mathcal{P}$ are perfect, each graph $P \in \mathcal{P}$ can be colored with $\omega(P) \leq \omega(G)$ colors; we may assume that the sets of colors used on the graphs in $\mathcal{P}$ are pairwise disjoint. Now we take the union of the colorings of the graphs in $\mathcal{P}$ to obtain a coloring of $G$ that uses at most $\omega(G)^2$ colors. 
\end{proof} 

\section{Necklaces}
\label{sec:necklaces}

We begin with some definitions. Let $n$ be a non-negative integer, and let $m_0,...,m_n$ be positive integers. Let $H$ be a graph whose vertex-set is $\bigcup_{i=0}^n \{x_{i,0},x_{i,1},...,x_{i,m_i}\} \cup \{y_1,...,y_n\}$, with adjacency as follows: 
\begin{itemize} 
\item $x_{0,0}-...-x_{0,m_0}-x_{1,0}-...-x_{1,m_1}-...-x_{n,0}-...-x_{n,m_n}$ is a chordless path; 
\item $\{y_1,...,y_n\}$ is a stable set; 
\item for all $i \in \{1,...,n\}$, the vertex $y_i$ has exactly two neighbors in the set $\bigcup_{i=0}^n \{x_{i,0},x_{i,1},...,x_{i,m_i}\}$, namely $x_{i-1,m_{i-1}}$ and $x_{i,0}$. 
\end{itemize} 
Under these circumstances, we say that $H$ is an {\em $(m_0,...,m_n)$-necklace with base $x_{0,0}$ and hook $x_{n,m_n}$}, or simply that $H$ is an {\em $(m_0,...,m_n)$-necklace}. If $G$ is a subdivision of $H$, then we say that $G$ is an {\em $(m_0,...,m_n)$-necklace$^*$ with base $x_{0,0}$ and hook $x_{n,m_n}$}, or simply that $G$ is an {\em $(m_0,...,m_n)$-necklace$^*$}. To simplify notation, given a non-negative integer $n$ and a positive integer $m$, we often write ``$(m)_n$-necklace'' instead of ``$\underbrace{(m,...,m)}_{n+1}$-necklace,'' and ``$(m)_n$-necklace$^*$'' instead of ``$\underbrace{(m,...,m)}_{n+1}$-necklace$^*$.'' (We remark that a $(1)_1$-necklace is the bull, and that for all positive integers $m$, an $(m)_0$-necklace with base $x_0$ and hook $x_m$ is a chordless $m$-edge path between $x_0$ and $x_m$.) 
\\
\\
Our goal in this section is to prove that for all non-negative integers $n$ and positive integers $m_0,...,m_n$, the class ${\rm Forb}^*((m_0,...,m_n)-necklace)$ is $\chi$-bounded by an exponential function (see \ref{necklace} below). We observe that in order to prove \ref{necklace}, it suffices to consider only the $(m)_n$-necklaces. Indeed, if $m = \max\{m_0,...,m_n\}$, then an $(m)_n$-necklace is a subdivision of an $(m_0,...,m_n)$-necklace, and consequently, ${\rm Forb}^*((m_0,...,m_n)-necklace) \subseteq {\rm Forb}^*((m)_n-necklace)$. Thus, it suffices to show that ${\rm Forb}^*((m)_n-necklace)$ is $\chi$-bounded by an exponential function. 
\\
\\
We now need some more definitions. First, in this paper, the {\em local chromatic number} of a graph $G$, denoted by $\chi_l(G)$, is the number $\max_{v \in V_G} \chi(G[\Gamma_G(v)])$. Next, let $n$ be a non-negative and $m$ a positive integer. Let $G$ be a graph whose vertex-set is the disjoint union of non-empty sets $N$ and $X$, let $x_0$ and $x$ be distinct vertices in $N$, and assume that the adjacency in $G$ is as follows: 
\begin{itemize} 
\item $G[N]$ is an $(m)_n$-necklace$^*$ with base $x_0$ and hook $x$; 
\item $G[X]$ is connected; 
\item $N \smallsetminus \{x\}$ is anti-complete to $X$; 
\item $x$ has a neighbor in $X$. 
\end{itemize} 
Under these circumstances, we say that $(G,x_0,x)$ is an {\em $(m)_n$-alloy} or simply an {\em alloy}. The graph $G$ is referred to as the {\em base graph} of the alloy $(G,x_0,x)$, and the ordered pair $(N,X)$ is the {\em partition} of the alloy $(G,x_0,x)$. The {\em potential} of the alloy $(G,x_0,x)$ is the chromatic number of the graph $G[X]$. 
\\
\\
We now state the main technical lemma of this section. 
\begin{theorem} \label{alloy} Let $G$ be a connected graph, and let $x_0 \in V_G$. Let $n$ and $\beta$ be non-negative integers, and let $m$ and $\alpha$ be positive integers. Assume that $\chi_l(G) \leq \alpha$ and $\chi(G) > 2^{n+1}((m+3)\alpha+\beta)$. Then there exists an induced subgraph $H$ of $G$ and a vertex $x \in V_G$ such that $(H,x_0,x)$ is an $(m)_n$-alloy of potential greater than $\beta$. 
\end{theorem} 
\noindent 
Since the base graph of an $(m)_n$-alloy contains an $(m)_n$-necklace$^*$ as an induced subgraph, \ref{alloy} easily implies the main result of this section (\ref{necklace}), as we now show. (We note that our proof of \ref{necklace} relies only on the special case of \ref{alloy} when $\beta = 0$.) 
\begin{theorem} \label{necklace} Let $n$ be a non-negative integer, let $m_0,...,m_n$ be positive integers, and let $m = \max\{m_0,...,m_n\}$. Then the class ${\rm Forb}^*((m_0,...,m_n)-necklace)$ is $\chi$-bounded by the exponential function $f(k) = (2^{n+1}(m+3))^{k-1}$. 
\end{theorem} 
\begin{proof} 
Since an $(m)_n$-necklace is a subdivision of an $(m_0,...,m_n)$-necklace, we know that ${\rm Forb}^*((m_0,...,m_n)-necklace) \subseteq {\rm Forb}^*((m)_n-necklace)$, and so it suffices to show that ${\rm Forb}^*((m)_n-necklace)$ is $\chi$-bounded by the function $f$. Suppose that this is not the case, and let $k \in \mathbb{N}$ be minimal with the property that there exists a graph $G \in {\rm Forb}^*((m)_n-necklace)$ such that $\omega(G) = k$ and $\chi(G) > f(k)$. Clearly, $k \geq 2$. Furthermore, we may assume that $G$ is connected, for otherwise, instead of $G$, we consider a component of $G$ with maximum chromatic number. Note that for all $v \in V_G$, we have that $\omega(G[\Gamma_G(v)]) \leq k-1$, and so by the minimality of $k$, $\chi(G[\Gamma_G(v)]) \leq f(k-1)$; thus $\chi_l(G) \leq f(k-1)$. Now, set $\alpha = f(k-1)$; then $\chi_l(G) \leq \alpha$ and $\chi(G) > 2^{n+1}(m+3)\alpha$. Fix $x_0 \in V_G$. Then \ref{alloy} implies that there exists an induced subgraph $H$ of $G$ and a vertex $x \in V_G$ such that $(H,x_0,x)$ is an $(m)_n$-alloy. But then $H$ contains an $(m)_n$-necklace$^*$ as an induced subgraph, contrary to the fact that $G \in {\rm Forb}^*((m)_n-necklace)$. 
\end{proof} 
\noindent 
The rest of the section is devoted to proving \ref{alloy}. The idea of the proof is to show that, given a connected graph $G$ whose chromatic number is sufficiently large relative to its local chromatic number, it is possible to recursively ``chisel'' an $(m)_n$-alloy out of the graph $G$. At each recursive step, the ``length'' of the alloy (i.e.\ the number $n$) increases, and the potential of the alloy decreases (but in a controlled fashion, so as to allow the next recursive step). We begin with a technical lemma, which we will use many times in this section. 
\begin{theorem} \label{path lemma} Let $G$ be a graph, let $x_0 \in V_G$, and let $S \subseteq V_G \smallsetminus \{x_0\}$ be such that $G[S]$ is connected and $x_0$ has a neighbor in $S$. Let $k$ be a non-negative integer, let $\alpha$ be a positive integer, and assume that $\chi_l(G) \leq \alpha$, and that $\chi(G[S]) > k\alpha$. Then there exist vertices $x_1,...,x_k \in S$ and a set $X \subseteq S$ such that: 
\begin{itemize} 
\item[a.] $x_0-x_1-...-x_k$ is an induced path in $G$; 
\item[b.] $G[X]$ is connected; 
\item[c.] $x_1,...,x_k \notin X$; 
\item[d.] $x_k$ has a neighbor in $X$; 
\item[e.] vertices $x_0,...,x_{k-1}$ are anti-complete to $X$; 
\item[f.] $\chi(G[X]) \geq \chi(G[S])-k\alpha$. 
\end{itemize} 
\end{theorem} 
\begin{proof} 
Let $i \in \{0,...,k\}$ be maximal such that there exist vertices $x_1,...,x_i \in S$ and a set $X \subseteq S$ such that: 
\begin{itemize} 
\item $x_0-x_1-...-x_i$ is an induced path in $G$; 
\item $G[X]$ is connected; 
\item $x_1,...,x_i \notin X$; 
\item $x_i$ has a neighbor in $X$; 
\item vertices $x_0,...,x_{i-1}$ are anti-complete to $X$; 
\item $\chi(G[X]) \geq \chi(G[S])-i\alpha$. 
\end{itemize} 
(The existence of such an index $i$ follows from the fact that $x_0$ is an induced path in $G$, $G[S]$ is connected, $x_0$ has a neighbor in $S$, and $\chi(G[S]) \geq \chi(G[S])-0 \cdot \alpha$.) 
\\
\\
We need to show that $i = k$. Suppose otherwise, that is, suppose that $i < k$. Then: 
\begin{displaymath} 
\begin{array}{rcl} 
\chi(G[X]) & \geq & \chi(G[S])-i\alpha 
\\
& > & k\alpha-i\alpha 
\\
& = & (k-i)\alpha 
\\
& \geq & \alpha, 
\end{array} 
\end{displaymath} 
\noindent 
and so $\chi(G[X]) > \alpha$. Since $\chi(G[\Gamma_G(x_i)]) \leq \alpha$ (because $\chi_l(G) \leq \alpha$), it follows that $x_i$ is not complete to $X$; let $X'$ be the vertex-set of a component of $G[X \smallsetminus \Gamma_G(x_i)]$ with maximum chromatic number. Then $\chi(G[X]) \leq \chi(G[\Gamma_G(x_i)])+\chi(G[X'])$, and so: 
\begin{displaymath} 
\begin{array}{rcl} 
\chi(G[X']) & \geq & \chi(G[X])-\chi(G[\Gamma_G(x_i)]) 
\\
& \geq & (\chi(G[S])-i\alpha)-\alpha 
\\
& = & \chi(G[S])-(i+1)\alpha 
\end{array} 
\end{displaymath} 
\noindent 
Fix a vertex $x_{i+1} \in X \cap \Gamma_G(x_i)$ such that $x_{i+1}$ has a neighbor in $X'$. But now the sequence $x_1,...,x_i,x_{i+1}$ and the set $X'$ contradict the maximality of $i$. It follows that $i = k$, which completes the argument. 
\end{proof} 
\noindent 
The following is an easy consequence of \ref{path lemma}, and it will serve as the base for our recursive construction of an $(m)_n$-alloy. 
\begin{theorem} \label{alloy base case} Let $G$ be a connected graph, let $x_0 \in V_G$, let $\beta$ be a non-negative integer, and let $m$ and $\alpha$ be positive integers. Assume that $\chi_l(G) \leq \alpha$, and that $\chi(G) > (m+1)\alpha+\beta$. Then there exists a vertex $x \in V_G \smallsetminus \{x_0\}$ and an induced subgraph $H$ of $G$ such that $(H,x_0,x)$ is an $(m)_0$-alloy of potential greater than $\beta$. 
\end{theorem} 
\begin{proof} 
Let $S$ be the vertex-set of a component of $G \smallsetminus x_0$ of maximum chromatic number. Clearly then, $\chi(G) \leq \chi(G[S])+1$, and consequently, $\chi(G[S]) > m\alpha+\beta$. Since $G$ is connected, $x_0$ has a neighbor in $S$. By \ref{path lemma} then, there exist vertices $x_1,...,x_m \in S$ and a set $X \subseteq S$ such that: 
\begin{itemize} 
\item $x_0-x_1-...-x_m$ is an induced path in $G$; 
\item $G[X]$ is connected; 
\item $x_1,...,x_m \notin X$; 
\item $x_m$ has a neighbor in $X$; 
\item vertices $x_0,...,x_{m-1}$ are anti-complete to $X$; 
\item $\chi(G[X]) \geq \chi(G[S])-m\alpha$. 
\end{itemize} 
The fact that $\chi(G[X]) \geq \chi(G[S])-m\alpha$ and $\chi(G[S]) > m\alpha+\beta$ implies that $\chi(G[X]) > \beta$. Now set $H = G[\{x_0,...,x_{m_0}\} \cup X]$ and $x = x_m$. Then $(H,x_0,x)$ is an $(m)_0$-alloy of potential greater than $\beta$. 
\end{proof} 
\noindent 
Our goal now is to show that, given an $(m)_n$-alloy with large potential and small local chromatic number of the base graph, we can ``chisel'' out of this $(m)_n$-alloy an $(m)_{n+1}$-alloy of large potential. More formally, we wish to prove the following lemma. 
\begin{theorem} \label{alloy induction case} Let $n$ and $\beta$ be non-negative integers, and let $m$ and $\alpha$ be positive integers. Let $(G,x_0,x)$ be an $(m)_n$-alloy of potential greater than $2((m+3)\alpha+\beta)$, and let $(N,X)$ be the partition of the alloy $(G,x_0,x)$. Assume that $\chi_l(G) \leq \alpha$. Then there exist disjoint sets $N',X' \subseteq V_G$ such that $N \subseteq N'$ and $X' \subseteq X$, and a vertex $x' \in X$ such that $(G[N' \cup X'],x_0,x')$ is an $(m)_{n+1}$-alloy of potential greater than $\beta$ and with partition $(N',X')$. 
\end{theorem} 
\noindent 
We now need some definitions. Let $n$ be a non-negative and $m$ a positive integer, and let $(G,x_0,x)$ be an $(m)_n$-alloy with partition $(N,X)$. Assume that the potential of $(G,x_0,x)$ is greater than $2\beta$ (where $\beta$ is some non-negative integer). For each $i \in \mathbb{N} \cup \{0\}$, let $S_i'$ be the set of all vertices in $\{x\} \cup X$ that are at distance $i$ from $x$ in $G[\{x\} \cup X]$; thus, $S_0' = \{x\}$. Let $t \in \mathbb{N}$ be such that $\chi(G[S_t'])$ is as large as possible. As the sets $S_1,S_3,S_5,...$ are pairwise anti-complete to each other, as are the sets $S_2,S_4,S_6,...$, it is easy to see that $\chi(G[X]) \leq 2\chi(G[S_t'])$, and consequently, $\chi(G[S_t']) > \beta$. Now, let $S_t$ be the vertex-set of a component of $G[S_t']$ with maximum chromatic number (thus, $\chi(G[S_t]) > \beta$), and for each $i \in \{0,1,...,t-1\}$, let $S_i$ be an inclusion-wise minimal subset of $S_i'$ such that every vertex in $S_{i+1}$ has a neighbor in $S_i$; clearly, $S_0 = \{x\}$. Let $H = G[N \cup \bigcup_{i=1}^t S_i]$. We then say that $(H,x_0,x)$ is a {\em reduction} of the $(m)_n$-alloy $(G,x_0,x)$, and that $\{S_i\}_{i=0}^t$ is the {\em stratification} of $(H,x_0,x)$. Clearly, $(H,x_0,x)$ is itself an $(m)_n$-alloy, and $(N,\bigcup_{i=1}^t S_i)$ is the associated partition. Further, as $\chi(G[S_t]) > \beta$ and $H$ is an induced subgraph of $G$, we know that $\chi(H[S_t]) > \beta$. Next, given vertices $a \in S_p$ and $b \in S_q$ for some $p,q \in \{0,...,t\}$, a path $P$ in $H$ between $a$ and $b$ is said to be {\em monotonic} provided that it has $|p-q|$ edges. This means that if $p = q$ then $a = b$, and if $p \neq q$ then all the internal vertices of the path $P$ lie in $\bigcup_{r=\min\{p,q\}+1}^{\max\{p,q\}-1} S_r$, with each set $S_r$ (with $\min\{p,q\}+1 \leq r \leq \max\{p,q\}-1$) containing exactly one vertex of the path. Clearly, every monotonic path is induced. We observe that for all $p \in \{0,...,t\}$ and $a \in S_p$, there exists a monotonic path between $x$ and $a$. 
\\
\\
The idea of the proof of \ref{alloy induction case} is as follows. First, we let $(H,x_0,x)$ be a reduction of the $(m)_n$-alloy $(G,x_0,x)$, and we let $\{S_i\}_{i=0}^t$ be the associated stratification. From now on, we work only with the graph $H$ (and not $G$). We find the needed vertex $x'$ in the set $S_t$, and the set $X'$ is chosen to be a suitable subset of the set $S_t$. Our proof splits into two cases. The first (and easier) case is when at least one of the sets $S_1,...,S_{t-2}$ is not stable (in this case, we necessarily have $t \geq 3$); the second (and harder) case is when the sets $S_1,...,S_{t-2}$ are all stable. We treat these two cases in two separate lemmas (the first case is treated in \ref{alloy easy}, and the second case in \ref{alloy hard}). 
\begin{theorem} \label{alloy easy} Let $n$ and $\beta$ be non-negative integers, and let $m$ and $\alpha$ be positive integers. Let $(G,x_0,x)$ be an $(m)_n$-alloy of potential greater than $2(m\alpha+\beta)$, and let $(N,X)$ be the partition of the alloy $(G,x_0,x)$. Assume that $\chi_l(G) \leq \alpha$. Let $(H,x_0,x)$ be a reduction of the $(m)_n$-alloy $(G,x_0,x)$, and let $\{S_i\}_{i=0}^t$ be the associated stratification. Assume that $t \geq 3$ and that at least one of the sets $S_1,...,S_{t-2}$ is not stable. Then there exist disjoint sets $N',X' \subseteq V_H$ such that $N \subseteq N'$ and $X' \subseteq S_t$, and a vertex $x' \in S_t$ such that $(H[N' \cup X'],x_0,x')$ is an $(m)_{n+1}$-alloy of potential greater than $\beta$ and with partition $(N',X')$. 
\end{theorem} 
\begin{proof} 
First, as pointed out above, we know that $\chi(H[S_t]) > m\alpha+\beta$. Now, let $r \in \{1,...,t-2\}$ be minimal with the property that $S_r$ is not stable; fix adjacent $a,b \in S_r$. Let $p \in \{0,...,r-1\}$ be maximal with the property that there exists some $z \in S_p$ such that for each $d \in \{a,b\}$, there exists a monotonic path $P_d$ between $z$ and $d$ (such an index $p$ and a vertex $z$ exist because $x_0 \in S_0$ and there exist monotonic paths between $x_0$ and $a$ and between $x_0$ and $b$). Since $S_0,...,S_{r-1}$ are all stable, this means that $H[V_{P_a} \cup V_{P_b}]$ is a chordless cycle, and by construction, $(V_{P_a} \cup V_{P_b}) \cap S_p = \{z\}$ and $(V_{P_a} \cup V_{P_b}) \cap S_r = \{a,b\}$. Next, let $Q$ be a monotonic path between $x$ and $z$. By the minimality of $S_r$, there exists some $s_{r+1} \in S_{r+1}$ that is adjacent to $a$ and non-adjacent to $b$. Now, fix some $s_{t-1} \in S_{t-1}$ such that there exists a monotonic path $R$ between $s_{r+1}$ and $s_{t-1}$ (the existence of $s_{t-1}$ follows from the fact that for all $i \in \{0,...,t-1\}$ and $v \in S_i$, $v$ has a neighbor in $S_{i+1}$). Since $s_{t-1}$ has a neighbor in $S_t$, and since $\chi(H[S_t]) > m\alpha$, we can apply \ref{path lemma} to the vertex $s_{t-1}$ and the set $S_t$ to obtain vertices $u_1,...,u_m \in S_t$ and a set $X' \subseteq S_t \smallsetminus \{u_1,...,u_m\}$ such that the following hold: 
\begin{itemize} 
\item $s_{t-1}-u_1-...-u_m$ is an induced path in $G$; 
\item $u_m$ has a neighbor in $X'$; 
\item vertices $s_{t-1},u_1,...,u_{m-1}$ are anti-complete to $X'$; 
\item $H[X']$ is connected; 
\item $\chi(H[X']) \geq \chi(H[S_t])-m\alpha$. 
\end{itemize} 
Set $N' = N \cup V_{Q} \cup V_{P_a} \cup V_{P_b} \cup V_R \cup \{u_1,...,u_m\}$ and $x' = u_m$. Clearly then, $(H[N' \cup X'],x_0,x')$ is an $(m)_{n+1}$-alloy with partition $(N',X')$. Since $\chi(H[X']) \geq \chi(H[S_t])-m\alpha$ and $\chi(H[S_t]) > m\alpha+\beta$, we get that $\chi(H[X']) > \beta$. This completes the argument. 
\end{proof} 
\begin{theorem} \label{alloy hard} Let $n$ and $\beta$ be non-negative integers, and let $m$ and $\alpha$ be positive integers. Let $(G,x_0,x)$ be an $(m)_n$-alloy of potential greater than $2((m+3)\alpha+\beta)$, and let $(N,X)$ be the partition of the alloy $(G,x_0,x)$. Assume that $\chi_l(G) \leq \alpha$. Let $(H,x_0,x)$ be a reduction of the $(m)_n$-alloy $(G,x_0,x)$, and let $\{S_i\}_{i=0}^t$ be the associated stratification. Assume that the sets $S_1,...,S_{t-2}$ are all stable. Then there exist disjoint sets $N',X' \subseteq V_H$ such that $N \subseteq N'$ and $X' \subseteq S_t$, and a vertex $x' \in S_t$ such that $(H[N' \cup X'],x_0,x')$ is an $(m)_{n+1}$-alloy of potential greater than $\beta$ and with partition $(N',X')$. 
\end{theorem} 
\begin{proof} 
First, since the potential of the alloy $(G,x_0,x)$ is greater than $2((m+3)\alpha+\beta)$, we know that $\chi(H[S_t]) > (m+3)\alpha+\beta$. Next, fix $a \in S_{t-1}$, and set $A = S_t \cap \Gamma_H(a)$. Note that $\chi(H[S_t]) > 2\alpha$, and so we can apply \ref{path lemma} to the vertex $a$ and the set $S_t$ in $H$ to obtain vertices $u_0',u_1' \in S_t$ and a non-empty set $C \subseteq S_t \smallsetminus \{u_0',u_1'\}$ such that $a-u_0'-u_1'$ is an induced path in $H$, $a$ and $u_0'$ are anti-complete to $C$ (note that this implies that $C \cap A = \emptyset$), $u_1'$ has a neighbor in $C$, $H[C]$ is connected, and 
\begin{displaymath} 
\begin{array}{rcl} 
\chi(H[C]) & \geq & \chi(H[S_t])-2\alpha 
\\
& > & ((m+3)\alpha+\beta)-2\alpha 
\\
& = & (m+1)\alpha+\beta. 
\end{array} 
\end{displaymath} 
Now, fix some $b \in S_{t-1}$ adjacent to $u_1'$; since $a$ is not adjacent to $u_1'$, this means that $a \neq b$. Set $B = S_t \cap \Gamma_H(b)$; clearly, $u_1' \in B$. Since $\chi(H[C]) > \alpha$ and $\chi(H[B]) \leq \alpha$, we know that $C \not\subseteq B$; let $U$ be the vertex-set of a component of $H[C \smallsetminus B]$ with maximum chromatic number. Then 
\begin{displaymath} 
\begin{array}{rcl} 
\chi(H[C]) & \leq & \chi(H[B])+\chi(H[U]) 
\\
& \leq & \alpha+\chi(H[U]), 
\end{array} 
\end{displaymath} 
and so $\chi(H[U]) > m\alpha+\beta$. Note that by construction, neither $A$ nor $B$ intersects $U$. 
\\
\\
Let us define a path of {\em type one} in $H$ to be an induced path $u_0-...-u_p$ (with $p \geq 1$) in $H[S_t \smallsetminus U]$ such that $u_0 \in A \cup B$, exactly one vertex among $u_1,...,u_p$ is in $A \cup B$, $u_p$ has a neighbor in $U$, and $u_0,...,u_{p-1}$ are all anti-complete to $U$. We define a path of {\em type two} in $H$ to be an induced path $u_0-...-u_p$ (with $p \geq 1$) in $H[S_t \smallsetminus U]$ such that $u_0 = u_0'$, no vertex among $u_1,...,u_p$ lies in $A \cup B$ (in particular, $u_1' \notin \{u_1,...,u_p\}$), $u_p$ has a neighbor in $U$, vertices $u_0,...,u_{p-1}$ are all anti-complete to $U$, and $u_1'$ is complete to $\{u_0,u_1\}$ and anti-complete to $\{u_2,...,u_p\} \cup U$. 
\\
\\
Our goal now is to show that $H$ contains a path of type one or two. Suppose that there is no path of type one in $H$. Since $H[S_t]$ is connected, and $u_0'$ is anti-complete to $U$, there exists an induced path $u_0-...-u_p$ (with $p \geq 1$) in $H[S_t \smallsetminus U]$ such that $u_0 = u_0'$, $u_p$ has a neighbor in $U$, and vertices $u_0,...,u_{p-1}$ are anti-complete to $U$. Note that $u_0 \in A$ (because $u_0 = u_0'$ and $u_0' \in A$). Clearly then, $u_1,...,u_p \notin A \cup B$, for otherwise, at least two vertices among $u_0,u_1,...,u_p$ would lie in $A \cup B$, and then $u_{p'}-u_{p'+1}-...-u_p$ would be a path of type one in $H$ for $p' \in \{0,...,p-1\}$ chosen maximal with the property that at least two vertices among $u_{p'},u_{p'+1},...,u_p$ lie in $A \cup B$. Since $u_0 = u_0'$ and $u_1,...,u_p \notin A \cup B$, we know that $u_1' \notin \{u_0,...,u_p\}$. Next, note that $u_1'$ is anti-complete to $U$, for otherwise, $u_0'-u_1'$ would be a path of type one in $H$. Further, $u_1'$ is anti-complete to $\{u_2,...,u_p\}$, for otherwise, we let $p' \in \{2,...,p\}$ be maximal with the property that $u_1'$ is adjacent to $u_{p'}$, and we observe that $u_0'-u_1'-u_{p'}-u_{p'+1}-...-u_p$ is a path of type one in $H$. Finally, $u_1'$ is adjacent to $u_1$, for otherwise, $u_1'-u_0-u_1-...-u_p$ would be a path of type one in $H$. Thus, $u_0-...-u_p$ is a path of type two in $H$. This proves that $H$ contains a path of type one or two. 
\\
\\
Let $u_0-...-u_p$ (with $p \geq 1$) be a path of type one or two in $H$. Recall that $\chi(H[U]) > m\alpha+\beta$. We now apply \ref{path lemma} to the vertex $u_p$ and the set $U$ in $H$ to obtain vertices $u_{p+1},...,u_{p+m} \in U$ and a set $X' \subseteq U \smallsetminus \{u_{p+1},...,u_{p+m}\}$ such that the following hold: 
\begin{itemize} 
\item $u_p-u_{p+1}-...-u_{p+m}$ is an induced path in $H$; 
\item $u_{p+m}$ has a neighbor in $X'$; 
\item vertices $u_p,...,u_{p+m-1}$ are anti-complete to $X'$; 
\item $H[X']$ is connected; 
\item $\chi(H[X']) \geq \chi(H[U])-m\alpha$; 
\end{itemize} 
note that the last condition, together with the fact that $\chi(H[U]) > m\alpha+\beta$, implies that $\chi(H[X']) > \beta$. Set $x' = u_{p+m}$. Our goal is to construct a set $N'$ with $N \subseteq N'$ such that $(H[N' \cup X'],x_0,x')$ is an $(m)_{n+1}$-alloy with partition $(N',X')$. Since $\chi(H[X']) > \beta$, the potential of any such alloy is greater than $\beta$, as desired. 
\\
\\
First, if $u_0-...-u_p$ is a path of type two in $H$, then we let $P$ be a monotonic path between $a$ and $x$, we set $N' = N \cup V_P \cup \{u_0,...,u_{p+m}\} \cup \{u_1'\}$, and we are done. From now on, we assume that $u_0-...-u_p$ is a path of type one in $H$. Fix $l \in \{1,...,p\}$ such that $u_l \in A \cup B$; then by the definition of a path of type one in $H$, we get that $u_0,u_l \in A \cup B$, and no other vertex on the path $u_0-...-u_p$ lies in $A \cup B$. If some vertex $d \in \{a,b\}$ is complete to $\{u_0,u_l\}$, then we let $P$ be a monotonic path between $x$ and $d$, we set $N' = N \cup V_P \cup \{u_0,...,u_{p+m}\}$, and we are done. From now on, we assume that neither $a$ nor $b$ is complete to $\{u_0,u_l\}$. Then one of $a$ and $b$ is adjacent to $u_0$ and non-adjacent to $u_l$, and the other is adjacent to $u_l$ and non-adjacent to $u_0$. Now, fix maximal $q \in \{0,...,t-2\}$ such that there exists a vertex $z \in S_q$ with the property that for each $d \in \{a,b\}$, there exists a monotonic path $P_d$ between $z$ and $d$. Since $S_0,...,S_{t-2}$ are all stable, we get that if $a$ and $b$ are adjacent then $H[V_{P_a} \cup V_{P_b}]$ is a chordless cycle, and if $a$ and $b$ are non-adjacent then $H[V_{P_a} \cup V_{P_b}]$ is an induced path between $a$ and $b$; in either case, we have that $(V_{P_a} \cup V_{P_b}) \cap S_{t-1} = \{a,b\}$ and $(V_{P_a} \cup V_{P_b}) \cap S_q = \{z\}$. Let $Q$ be a monotonic path between $z$ and $x$. Now, if $a$ and $b$ are adjacent, then we set $N' = N \cup V_Q \cup V_{P_a} \cup V_{P_b} \cup \{u_l,u_{l+1},...,u_{p+m}\}$; and if $a$ and $b$ are non-adjacent, then we set $N' = V_Q \cup V_{P_a} \cup V_{P_b} \cup \{u_0,...,u_{p+m}\}$. This completes the argument. 
\end{proof} 
\noindent 
We can now prove \ref{alloy induction case}, restated below. 
\begin{alloy induction case} Let $n$ and $\beta$ be non-negative integers, and let $m$ and $\alpha$ be positive integers. Let $(G,x_0,x)$ be an $(m)_n$-alloy of potential greater than $2((m+3)\alpha+\beta)$, and let $(N,X)$ be the partition of the alloy $(G,x_0,x)$. Assume that $\chi_l(G) \leq \alpha$. Then there exist disjoint sets $N',X' \subseteq V_G$ such that $N \subseteq N'$ and $X' \subseteq X$, and a vertex $x' \in X$ such that $(G[N' \cup X'],x_0,x')$ is an $(m)_{n+1}$-alloy of potential greater than $\beta$ and with partition $(N',X')$. 
\end{alloy induction case} 
\begin{proof} 
Let $(H,x_0,x)$ be a reduction of the $(m)_n$-alloy $(G,x_0,x)$, and let $\{S_i\}_{i=0}^t$ be the associated stratification. If $t \geq 3$ and at least one of the sets $S_1,...,S_{t-2}$ is not stable, then the result follows from \ref{alloy easy}. Otherwise, the result follows from \ref{alloy hard}. 
\end{proof} 
\noindent 
Finally, we use \ref{alloy base case} and \ref{alloy induction case} to prove \ref{alloy}, restated below. 
\begin{alloy} Let $G$ be a connected graph, and let $x_0 \in V_G$. Let $n$ and $\beta$ be non-negative integers, and let $m$ and $\alpha$ be positive integers. Assume that $\chi_l(G) \leq \alpha$ and $\chi(G) > 2^{n+1}((m+3)\alpha+\beta)$. Then there exists an induced subgraph $H$ of $G$ and a vertex $x \in V_G$ such that $(H,x_0,x)$ is an $(m)_n$-alloy of potential greater than $\beta$. 
\end{alloy} 
\begin{proof} 
For all $j \in \{0,...,n\}$, set $\beta_j = \beta+(\Sigma_{i=1}^{n-j} 2^i)((m+3)\alpha+\beta)$. Our goal is to prove inductively that for all $j \in \{0,...,n\}$, there exist disjoint sets $N_j,X_j \subseteq V_G$ and a vertex $x^j \in V_G$ such that $(G[N_j \cup X_j],x_0,x^j)$ is an $(m)_j$-alloy of potential greater than $\beta_j$. Since $\beta_n = \beta$, the result will follow. 
\\
\\
For the base case (when $j=0$), we observe that 
\begin{displaymath} 
\begin{array}{rcl} 
\chi(G) & > & 2^{n+1}((m+3)\alpha+\beta) 
\\
& > & (\Sigma_{i=0}^n 2^i)((m+3)\alpha+\beta) 
\\
& = & (m+3)\alpha+\beta+(\Sigma_{i=1}^n 2^i)((m+3)\alpha+\beta) 
\\
& = & (m+3)\alpha+\beta_0 
\\
& > & (m+1)\alpha+\beta_0, 
\end{array} 
\end{displaymath} 
and so \ref{alloy base case} implies that there exist sets $N_0,X_0 \subseteq V_G$ and a vertex $x^0 \in V_G$ such that $(G[N_j \cup X_j],x_0,x^j)$ is an $(m)_j$-alloy of potential greater than $\beta_0$. 
\\
\\
For the induction case, suppose that $j \in \{0,...,n-1\}$ and that there exist disjoint sets $N_j,X_j \subseteq V_G$ and a vertex $x^j \in V_G$ such that $(G[N_j \cup X_j],x_0,x^j)$ is an $(m)_j$-alloy of potential greater than $\beta_j$. Since 
\begin{displaymath} 
\begin{array}{rcl} 
\beta_j & = & \beta+(\Sigma_{i=1}^{n-j} 2^i)((m+3)\alpha+\beta) 
\\
& \geq & (\Sigma_{i=1}^{n-j} 2^i)((m+3)\alpha+\beta) 
\\
& = & 2((m+3)\alpha+\beta+(\Sigma_{i=1}^{n-(j+1)} 2^i)((m+3)\alpha+\beta)) 
\\
& = & 2((m+3)\alpha+\beta_{j+1}), 
\end{array} 
\end{displaymath} 
\ref{alloy induction case} implies that there exist sets $N_{j+1},X_{j+1} \subseteq V_G$ and a vertex $x^{j+1}$ such that $(G[N_{j+1} \cup X_{j+1}],x_0,x^{j+1})$ is an $(m)_{j+1}$-alloy of potential greater than $\beta_{j+1}$. This completes the induction. 
\end{proof}

\end{document}